\documentclass[reqno]{amsart}
\usepackage{amssymb,amsmath,times}
\usepackage[mathscr]{eucal}
\usepackage{xspace}
\usepackage{stix}
\usepackage[colorlinks = false]{hyperref}

\newcommand{\R}{\mathbb{R}}

\newcommand{\C}{\mathbb{C}}

\newcommand{\po}{\partial}

\newcommand{\ve}{\varepsilon}

\newcommand{\la}{\left\langle}
\newcommand{\ra}{\right\rangle}

\newcommand{\loc}{{\text{\rm loc}}}
\newcommand{\X}{\times}

\renewcommand{\d}{\delta}
\renewcommand{\l}{\lambda}

\renewcommand{\a}{\alpha}

\newcommand{\g}{\gamma} 
\newcommand{\z}{\zeta}

\newcommand{\supp}{\text{\rm supp}\,}

\newcommand{\M}{{\mathcal M}}

\renewcommand{\supp}{\text{\rm supp}\,}

\newcommand{\bbT}{{\mathbb T}}

\newcommand{\uu}{{\frak u}}

\newcommand{\fka}{{\frak a}}
\newcommand{\fkb}{{\frak b}}

\theoremstyle{plain}
\newtheorem{theorem}{Theorem}[section]

\theoremstyle{definition}
\newtheorem{definition}{Definition}[section]
\theoremstyle{remark}

\newtheorem{remark}{Remark}[section]
\numberwithin{equation}{section}

\begin{document}

\title[Relativistic short wave-long wave interactions]
{On short wave-long wave interactions in the relativistic context.}

\author{Jo\~ao Paulo Dias}
\thanks{J.P.\ Dias gratefully acknowledges the support from FCT (Funda\c c\~ao para a Ci\^encia  e a Tecnologia)
under the project UIDB/04561/2020}
\address{Departamento de Matem\'atica and CMAFcIO \\
Faculdade de Ci\^encias, Universidade de Lisboa\\
Campo Grande, Edif.\ C6, 1749-016, Lisboa, Portugal }
\email{jpdias@fc.ul.pt}

\author{Hermano Frid} 
\thanks{H.~Frid gratefully acknowledges the support from CNPq, 
through grant proc.\ 305097/2019-9, and FAPERJ, 
through grant proc.\ E-26/202.900-2017}

\address{Instituto de Matem\'atica Pura e Aplicada - IMPA\\
         Estrada Dona Castorina, 110 \\
         Rio de Janeiro, RJ 22460-320, Brazil}
\email{hermano@impa.br}

\date{}



\begin{abstract}  In this paper we introduce models of short wave-long wave interactions in the relativistic setting. In this context the nonlinear Schr\"odinger equation is no longer adequate for describing short waves and  is replaced by a nonlinear Dirac equation. Two specific examples are considered: the case where the long waves are governed by a  scalar conservation law; and the case where the long waves are governed by the augmented Born-Infeld equations in electromagnetism.   

\end{abstract}

\maketitle 

\section{Introduction} \label{S:1}

In \cite{Be}, Benney proposed a general theory describing interactions between short waves and long waves, in the classical non-relativistic context. More specifically,  in Benney's model short waves are  described  by a non-linear Schr\" odinger equation. As for the long waves, in \cite{Be} two examples are given: a linear transport equation, and the Burgers equation, namely, with some simplifications, these examples are 
\begin{equation*}
\begin{aligned}
& i u_t + u_{xx}=|u|^2 u + \a v u\\
& v_t+ c_1 v_x=(\a|u|^2)_x,
\end{aligned}
\end{equation*}
and
\begin{equation*}
\begin{aligned}
& i u_t + u_{xx}=|u|^2 u + \a v u\\
& v_t+ (\frac{v^2}2)_x=(\a|u|^2)_x,
\end{aligned}
\end{equation*} 
where $\a>0$ is a constant.
We recall, among works dedicated to the study of this original model in \cite{Be},  that the well-posedness for the linear case was addressed in \cite{TH}, while the case of the Burgers equation with
dispersion, that is, the KdV equation, was  addressed in \cite{BOP}. In \cite{DF2}, global existence for the Burgers flux with a cubic perturbation, $a v^2- bv^3$, $b>0$, was obtained. We denote the coupling prescribed in \cite{Be}
by $\left(\begin{matrix} v u\\ |u|^2\end{matrix}\right)$. An important improvement in the model set forth in \cite{Be} was achieved in \cite{DFF} where the coupling, in the case where long waves are described by scalar conservation laws, was prescribed as           
$\left(\begin{matrix} g(v) u\\ g'(v) |u|^2\end{matrix}\right)$, where $\supp g'$ may be suitably chosen so as to guarantee the
preservation of the physical domain. Moreover, the improvement proposed in \cite{DFF} also enabled the study of interactions with long waves governed by systems of conservation laws such as elasticity, electromagnetism, symmetric systems, etc. 
It also opened the way for the study of interactions with compressible fluids in \cite{DF}, followed by extensions to heat conductive fluids and magnetohydrodynamics equations (see, e.g., \cite{FPZ, FJP, FMP, DRM, FMN}). An important feature in the latter references for interactions with fluids is that the nonlinear Schr\" odinger equation governing the short waves is based on the Lagrangian coordinates of the fluid. Also, the coupling in these works on interactions involving fluids has the form $ \left(\begin{matrix} g(v) h'(|u|^2)u\\ g'(v) h(|u|^2)\end{matrix}\right)$, with $\supp h'$ compact in $[0,\infty)$. 

\medskip
In the relativistic context, the short waves can no longer be described by a  nonlinear Schr\"odinger equation since this type of equation yields infinite  speed of propagation, which violates the relativity principle that no signal can propagate with speed higher than the speed of light. The natural substitute for the Schr\"odinger equation is the Dirac equation proposed by Dirac (\cite{Di}) in search of compatibility between relativity and quantum theories.  On the other hand as a replacement for the nonlinear cubic Schr\"odiger equation there are different models of 
the nonlinear cubic Dirac equation (see, e.g., \cite{CG,Ba, CH,Th, Del, KM, Ca}).  
Here we will be concerned with the Thirring model proposed by Thirring in \cite{Th} whose mathematical study has been considered in several papers (see, e.g., \cite{Del, KM, DF1, Ca}). 
More specifically, here we only consider the zero mass case.   

For instance, in the relativistic context, using the massless  Thirring model,  the  simplest case
of the transport equation found in \cite{Be} and recalled as the first system above   would be recast as  
\begin{equation}\label{e1.0}
\begin{aligned}
&  \uu_t = \fka\, \uu_{x}-  i(\l U+\a v)\uu,\\
& v_t+ c_1 v_x=(\a |\uu|^2)_x,\\
\end{aligned}
\end{equation}
\begin{equation}\label{e1.00}
U=\uu^\dag\uu-\uu^\dag\fka\uu\fka,
\end{equation}
where, as usual for Dirac equations, $\uu\in \C^2$ and $\fka$ is a $2\X 2$ complex matrix satisfying 
$\fka^*:=\bar{\fka}^\top=\fka$ and $\fka^2=I$. Here, $\a>0, \l\in\R$ and $U$ is the Thirring quadratic matrix valued functional, where ${}^\dag$ means the conjugate transpose, i.e., if $\uu=\left(\begin{matrix}u_1\\u_2\end{matrix}\right)\in\C^2$, then $\uu^\dag=(\bar u_1, \bar u_2)$, in particular, $\uu^\dag\uu=|\uu|^2=(\Re u_1)^2+(\Im u_1)^2+(\Re u_2)^2+(\Im u_2)^2$.  The justification of this type of model follows from the  justification for the corresponding model in the non-relativistic case in \cite{Be}. 

In this connection, we recall that Dias and Figueira in \cite{DF1} established  an important 
property of the solution of a simplified version of the massless Thirring model (with $U=|\uu|^2$), which is the fact that $|\uu|^2$ solves the wave equation (see also \cite{MO}).   Here we extend this property to a general massless Thirring model (with $U$
as in \eqref{e1.100})  with any real-valued potencial $V(t,x)$, in particular, that is, a nonlinear Dirac equation of the form
\begin{equation}\label{e1.000}
 \uu_t = \fka\, \uu_{x}-  i(\l U+ V)\uu,
 \end{equation}
with $V=V(t,x)$ any real-valued function, possibly depending on $\uu$. Not only $|\uu|^2$ satisfies the wave equation, but this is true also for $(\uu^\dag\fka\uu)$. This observation by itself trivializes the solution of the Cauchy problem for \eqref{e1.0}. We will prove this general property in Section~\ref{S:2}. 

In Sections~\ref{S:3} and~\ref{S:4} of this paper, we apply the result in Section~\ref{S:2} to two examples of models for  relativistic short wave-long wave interactions, for two different types of long waves
propagation. The first application, discussed in Section~\ref{S:3}, is the case of a  scalar conservation law in the relativistic context such as the one introduced by LeFloch, Makhlof and Okutmustur in \cite{LFMO} (see also \cite{HW}).  In this case the system modeling the short wave-long wave interactions has the form
\begin{align*}
& {\uu}_t=\fka\, {\uu}_x-i(\l U+ g(v)){\frak u}, \\
& v_t+ f(t,v)_x+ h(t,v)=\a (g'(v)|\uu|^2)_x,
\end{align*}
whose details will be explained in Section~\ref{S:3}. 

The second application, discussed in Section~\ref{S:4},   is provided by the augmented Born-Infeld (ABI) equations, introduced by Brenier in \cite{Br}, in electromagnetism. 
This is a linearly degenerate $8\X8$ system which shares some features with compressible fluid equations. More specifically, the two first equations of the system form themselves a closed system independent from the other remaining 6 variables and has the same structure of the equations for the so called Chapligyn gas (see, e.g., \cite{Se}). From this similarity with compressible gas dynamics, it is natural that the model for the referred interactions should be based on the Lagrangian coordinates of the ABI system. The system modeling the short wave-long wave interactions then reduces to the following, whose details are explained in Section~\ref{S:4} :
 \begin{align*}
&{\uu}_t=\fka\, {\frak u}_y-i(\l U+\a_1g_1(\theta)+\a_2g_2(\z))\uu , \\
& \theta_t-Z \theta_y=\a_1 (g_1'(\theta)|\uu|^2)_y,\\
&\z_t+Z\z_y=\a_2 (g_2'(\z)|\uu|^2)_y.
\end{align*}

We would like to remark here that an important property of the models of relativistic short wave-long wave interactions discussed in Sections~\ref{S:3} and~\ref{S:4} is their stability in the sense
 that if we have a sequence of weak solutions bounded in the natural norms, that is, $L^2$ for $\uu$ and  $L^\infty$ for the long waves, described by $v$ in the first case and by $(\theta,\z)$ in the second, then any weak limit of this sequence is also a weak solution of the corresponding system. This is a trivial consequence of the way we obtain the weak solutions and the entropy inequalities they satisfy. We state and prove this stability property in our main theorems in Sections~\ref{S:3} and~\ref{S:4}.

\section{The main property of the massless Thirring model.}\label{S:2}

In this section we state and prove the main property of the massless Thirring model for the purposes of this paper, which extends the corresponding fact proved in \cite{DF1} for the
homogeneous simplified version of the Thirring model. 

Consider the equation
\begin{equation}\label{e1.100}
\uu_t-\fka \uu_x=-iA(t,x)\uu
\end{equation}
where $\uu\in\C^2$,   $\fka$ is a  $2\X 2$ complex matrix satisfying  
$\fka^*:=(\bar\fka)^\top=\fka$, $\fka^2=I$, and $A$ is a $2\X2$ complex matrix satisfying $A^*=A$, $\fka A=A \fka$. An example of $A$ satisfying these conditions comes from the massless Thirring model with 
$A= \l U+ V$, where $U$ is as in \eqref{e1.00} and $V$ is any real-valued function. 

\begin{theorem}\label{T:100} Under the above conditions both  $w=|\uu|^2$  and 
$w=\uu^\dag \fka \uu$ satisfy the wave equation
\begin{equation}\label{e1.99}
w_{tt}-w_{xx}=0.
\end{equation}
\end{theorem}

\begin{proof} Then, multiplying \eqref{e1.100} by $\uu^\dag$ to the left
\begin{equation*}
\uu^\dag u_t-\uu^\dag\fka \uu_x=-i\uu^\dag A(t,x) \uu,
\end{equation*}
 applying ${}^\dag$ to the last equation
 $$
  \uu_t^\dag \uu  -\uu_x^\dag\fka \uu =i\uu^\dag A(t,x)\uu,
 $$
 adding the last two gives
 \begin{equation}\label{e1.100'}
 (|\uu|^2)_t-(\uu^\dag\fka \uu)_x=0.
 \end{equation}
Deriving the last equation by $t$, it follows
\begin{equation}\label{e1.102}
(|\uu|^2)_{tt}-(\uu^\dag\fka \uu)_{tx}=0.
\end{equation}
Similarly, multiplying \eqref{e1.100} by $\uu^\dag\fka$, it follows,
\begin{equation*}
\uu^\dag\fka \uu_t-\uu^\dag \uu_x=-i\uu^\dag\fka A(t,x)\uu,
\end{equation*}  
applying ${}^\dag$ to the last equation
\begin{equation*}
 u_t^\dag\fka \uu-\uu_x^\dag\uu=i\uu^\dag A(t,x) \fka \uu,
\end{equation*}  
adding the last two gives
\begin{equation}\label{e1.102'}
(\uu^\dag \fka \uu)_t-(|\uu|^2)_x=0.
 \end{equation}
Deriving the last equation by $x$, it follows
\begin{equation}\label{e1.103}
(\uu^\dag \fka \uu)_{xt}-(|\uu|^2)_{xx}=0.
\end{equation}
{}From \eqref{e1.102} and \eqref{e1.103} it follows
\begin{equation}\label{e1.104}
(|\uu|^2)_{tt}-(|\uu|^2)_{xx}=0.
\end{equation}
Similarly, deriving \eqref{e1.100'} with respect to $x$,  \eqref{e1.102'} with respect to $t$ and adding the
resulting equations we arrive at
\begin{equation}\label{e1.104'}
(\uu^\dag\fka\uu)_{tt}-(\uu^\dag\fka\uu)_{xx}=0,
\end{equation}
which proves the assertion for $\uu^\dag\fka\uu$ and concludes the proof. 

\end{proof}

\begin{remark}\label{R:1.1} Solving the wave equation, taking into account \eqref{e1.100'}, we obtain
\begin{multline}\label{e1.1000} 
|\uu|^2(t,x)=\frac12[|\uu(0,x+t)|^2+(\uu^\dag\fka\uu)(0,x+t)]\\
+\frac12[|\uu(0,x-t)|^2-(\uu^\dag\fka\uu)(0,x-t)].
\end{multline}
This formula shows  that $|\uu(t,x)|^2\ge0$, for all $t>0$,  as it should be, where we have used the fact that $|\fka \uu|=|\uu|$, and so $|(\uu^\dag\fka\uu)|\le |\uu|^2$. Moreover, if 
\begin{equation}\label{e1.IC}
|\uu(0,x)|^2 +(\uu^\dag\fka\uu)(0,x)>0,\quad \text{for all $x\in\R$}, 
\end{equation}
then $|\uu(t,x)|^2>0$, for all $t>0$. For instance, for 
$$
\fka=\left(\begin{matrix} 1&0\\0&-1\end{matrix}\right),
$$
as in \cite{Ca,KM},  if $\uu=(u_1,u_2)$, then  condition \eqref{e1.IC} is true if $|u_1(0,x)|^2>0$.

Similarly, using \eqref{e1.102'}, we arrive at the formula 
\begin{multline}\label{e1.1000'} 
(\uu^\dag\fka\uu)(t,x)=\frac12[|\uu(0,x+t)|^2+(\uu^\dag\fka\uu)(0,x+t)]\\
-\frac12[|\uu(0,x-t)|^2-(\uu^\dag\fka\uu)(0,x-t)].
\end{multline}

\end{remark} 

\begin{remark}\label{R:1.2} By the previous remark, in particular the formula \eqref{e1.1000}, 
the solution of the initial value problem for \eqref{e1.0} is trivial, as follows. We first solve the linear transport equation for $v$, with $|\uu|^2$ as a given right-hand side, obtained from \eqref{e1.1000}. Then, having found $v(t,x)$, we solve the equation for $\uu$, using the fact that \eqref{e1.1000} and \eqref{e1.1000'} determine $U$ everywhere, and so we can obtain   
$\uu$ by using the unitary group associated with the skew-adjoint operator $\fka\po_x$, with domain $H^1(\R)$, $S(t)=\operatorname{exp}\left(\fka \frac{\po}{\po x}\right)t$, and solving the Duhamel's equation by 
a standard  fixed point argument in $C([0,\infty); L^2(\R))$. We will give a bit more details about the solution of the equation for $\uu$ in the next section. 
\end{remark}

We take this opportunity to state and prove an extension of the above result to massless Dirac equations in three space dimensions as follows. 
Let us consider the equation  
\begin{equation}\label{e1.105}
\uu_t-\fka_1 \uu_x -\fka_2 \uu_y-\fka_3 \uu_z=-i B(t,x,y,z) \uu,
\end{equation}
where $\uu=\uu(t,x) \in\C^4$, $\fka_i$, $i=1,2,3$, are $4\X4$ complex matrices satisfying 
$\fka_i^*=\fka_i$, $\fka_i^2=I$, $\fka_i\fka_j=-\fka_j\fka_i$, $i\ne j$, $i,j=1,2,3$,
and  $B(t,x,y,z)$ is a $4\X4$ complex matrix such that $B^*=B$ and $\fka_i B=B\fka_i$, $i=1,2,3$.
Equations such as \eqref{e1.105} were proposed by R.T.~Glassey, as cited by Strauss in \cite{St}, p.245, where
$B=\l |\uu|^{p-1}I$, $p>1$, $\l\in\R$, and $I$ is the $4\X4$ identity matrix.  Also, \eqref{e1.105} includes a $1+3$-dimensional extension of the massless Thirring model, where 
$B(t,x,y,z)=\l\bar U+V(t,x,y,z)$ with $\bar U=\uu^\dag\uu-\uu^\dag\fkb\uu\fkb$  such that
$$
\fkb=i\fka_1\fka_2\fka_3,
$$
 and $V$ is a real-valued function.

\begin{theorem}\label{T:1.101} Let $\uu$ be a smooth solution of \eqref{e1.105} and let
$\fka_i$, $i=1,2,3$, satisfy the above properties. Then, both  $w=|\uu|^2$ and $w=\uu^\dag\fkb\uu$ satisfy
\begin{equation}\label{e1.89}
w_{tt}-w_{xx}-w_{yy}-w_{zz}=0.
\end{equation}
\end{theorem} 

\begin{proof} Multiplying \eqref{e1.105} by $\uu^\dag $ to the left
\begin{equation*}
\uu^\dag \uu_t-\uu^\dag \fka_1 \uu_x-\uu^\dag \fka_2 \uu_y-\uu^\dag\fka_3 \uu_z=-i\uu^\dag B(t,x,y,z)\uu,
\end{equation*}
 applying ${}^\dag$ to the last equation
 $$
   \uu_t^\dag \uu - \uu_x^\dag\fka_1 \uu- \uu_y^\dag\fka_2 \uu - \uu_z^\dag\fka_3 \uu=i\uu^\dag B(t,x,y,z) \uu,
 $$
 adding the last two gives
 $$
 (|\uu|^2)_t-(\uu^\dag\fka_1\uu)_x-(\uu^\dag\fka_2 \uu)_y-(\uu^\dag\fka_3 \uu)_z=0.
 $$
Deriving the last equation by $t$, it follows
\begin{equation}\label{e1.106}
(|\uu|^2)_{tt}-(\uu^\dag\fka_1 \uu)_{tx}-(\uu^\dag\fka_2 \uu)_{ty}-(\uu^\dag\fka_3 \uu)_{tz}=0.
\end{equation}
Similarly, multiplying \eqref{e1.105}  by $u^\dag \fka_1$,  it follows,
\begin{equation*}
\uu^\dag\fka_1 \uu_t-\uu^\dag \uu_x-\uu^\dag \fka_1\fka_2 \uu_y-\uu^\dag\fka_1\fka_3 \uu_z=
-iu^\dag B(t,x,y,z)\fka_1 u,
\end{equation*}  
applying ${}^\dag$ to the last equation
\begin{equation*}
 \uu_t^\dag \fka_1 \uu-\uu_x^\dag \uu- \uu_y^\dag\fka_2\fka_1  \uu-\uu_z^\dag\fka_3\fka_1  \uu=i\uu^\dag B(t,x,y,z)\fka_1\uu,
\end{equation*}  
adding the last two gives
\begin{equation}\label{e1.107}
(\uu^\dag \fka_1 \uu)_t-(|\uu|^2)_x=0.
\end{equation}
 Similarly, we get
 \begin{equation}\label{e1.108}
(\uu^\dag \fka_2 \uu)_t-(|\uu|^2)_y=0,
 \end{equation}
and
\begin{equation}\label{e1.109}
(\uu^\dag \fka_3 \uu)_t-(|\uu|^2)_z=0.
\end{equation}
Deriving \eqref{e1.107} by $x$, \eqref{e1.108} by $y$ and \eqref{e1.109} by $z$ there follow, respectively, 
\begin{equation}\label{e1.110}
(\uu^\dag \fka_1 \uu)_{xt}-(|\uu|^2)_{xx}=0,
\end{equation}
\begin{equation}\label{e1.111}
(\uu^\dag \fka_2 \uu)_{yt}-(|\uu|^2)_{yy}=0,
\end{equation}
\begin{equation}\label{e1.112}
(\uu^\dag \fka_3 \uu)_{zt}-(|\uu|^2)_{zz}=0.
\end{equation}
Adding \eqref{e1.106}, \eqref{e1.110}, \eqref{e1.111} and \eqref{e1.112}, it follows
\begin{equation}\label{e1.113}
( |\uu|^2)_{tt}-(|\uu|^2)_{xx}-(|\uu|^2)_{yy}-(|\uu|^2)_{zz}=0,
\end{equation}
which proves the assertion for $w=|\uu|^2$. To prove the assertion for $w=\uu^\dag\fkb\uu$, we first multiply \eqref{e1.105} by $\uu^\dag\fkb$ on the left to obtain
\begin{equation}\label{e1.114}
\uu^\dag\fkb\uu_t+\uu^\dag\fkb\fka_1\uu_x+\uu^\dag\fkb\fka_2\uu_y+\uu^\dag\fkb\fka_3\uu_z=-i\uu^\dag B\fkb \uu.
\end{equation}
 We then apply ${}^\dag$ to \eqref{e1.114} and add the resulting equation to \eqref{e1.114} to obtain
\begin{equation}\label{e1.115}
(\uu^\dag\fkb\uu)_t+(\uu^\dag\fkb\fka_1\uu)_x+(\uu^\dag\fkb\fka_2\uu)_y+(\uu^\dag\fkb\fka_3\uu)_z=0.
\end{equation}
Deriving \eqref{e1.115} by $t$ we obtain
\begin{equation}\label{e1.115'}
(\uu^\dag\fkb\uu)_{tt}+(\uu^\dag\fkb\fka_1\uu)_{tx}+(\uu^\dag\fkb\fka_2\uu)_{ty}+(\uu^\dag\fkb\fka_3\uu)_{tz}=0.
\end{equation}
Now we multiply \eqref{e1.105} by $\uu^\dag\fkb\fka_1$ to get
\begin{equation}\label{e1.116}
\uu^\dag\fkb\fka_1\uu_t+\uu^\dag\fkb\uu_x+\uu^\dag\fkb\fka_1\fka_2\uu_y+\uu^\dag\fkb\fka_1\fka_3\uu_z=-i\uu^\dag B\fkb\fka_1 \uu.
\end{equation}
We then apply ${}^\dag$ to \eqref{e1.116} and add the resulting equation to \eqref{e1.116}
to obtain
 \begin{equation*}
(\uu^\dag\fkb\fka_1\uu)_t+(\uu^\dag\fkb\uu)_x=0,
\end{equation*} 
which deriving with respect to $x$ gives
  \begin{equation}\label{e1.117}
(\uu^\dag\fkb\fka_1\uu)_{tx}+(\uu^\dag\fkb\uu)_{xx}=0.
\end{equation} 
Similarly, we obtain 
   \begin{equation}\label{e1.118}
(\uu^\dag\fkb\fka_2\uu)_{ty}+(\uu^\dag\fkb\uu)_{yy}=0,
\end{equation} 
and
  \begin{equation}\label{e1.119}
(\uu^\dag\fkb\fka_3\uu)_{tz}+(\uu^\dag\fkb\uu)_{zz}=0.
\end{equation} 
Adding \eqref{e1.115'}, \eqref{e1.117}, \eqref{e1.118} and \eqref{e1.119}  we then obtain \eqref{e1.89} for $w=\uu^\dag\fkb\uu$, which concludes the proof.

\end{proof}

\section{Application to Relativistic Scalar Conservation Laws}\label{S:3}
In this section we consider the interaction between short waves governed by a nonlinear massless Dirac equation and long waves governed by a scalar conservation law in the relativistic context such as the one proposed in \cite{LFMO}
(see also \cite{HW}). We consider  the following system  describing this interaction
\begin{align}
&{\frak u}_t= \fka\, {\frak u}_x-i(\l U+\a g(v))\frak u, \label{e1.1}\\
& v_t+\po_x f(t,v)+h(t,v)= \a (g'(v)|\uu|^2)_x,\label{e1.3}
\end{align}
where  $\uu=\left(\begin{matrix}u_1\\ u_2\end{matrix}\right)\in \C^2$, $\fka$ is a $2\X2$ complex matrix satisfying 
$\fka^\dag:=\bar\fka^T=\fka$, $\fka^2=I$, $\l\in\R$, $\a>0$ are constansts and $U$ is given in \eqref{e1.100}.  We assume that
$f,h\in C^2([0,\infty)\X\R)$, with $|f_v(t,v)|\le |v|$,  and $h(t,\pm c_0)=0$, $\pm h_v(t,\pm c_0)<0$,
respectively,  where $c_0$ is the speed of light, for all $t\ge0$.  

In \cite{HW} one has
$$
f(t,v)=\frac1{2a} v^2,\quad h(t,v)=\frac{\dot a}{a}v\left(1-\frac{v^2}{c_0^2}\right),
$$
where  $a\in C^2([0,\infty))$, $a(t)\ge 1$, $\dot a(t)>0$, for all $t\ge0$, $0<\d_1\le \dfrac{\dot a}{a}\le L_0$, and $c_0$ is the speed of light.

We prescribe   initial conditions
\begin{equation} \label{e1.4}
\uu(0,x)=\uu_0(x),\quad v(0,x)=v_0(x), \quad x\in\R,
\end{equation}
and we assume that
\begin{equation}\label{e1.4''}
\uu_0\in H^1(\R),\quad v_0\in L^\infty(\R)\cap L^1(\R),
\end{equation}
with
\begin{equation}\label{e1.5}
\|v_0\|_\infty< c_0.
\end{equation}

As to the function $g$ we assume that $g\in C^3(\R)$ and $\supp g'\subset (-M,M)$, with  $0<M<c_0$, satisfying
\begin{equation}\label{e1.4'}
|\{v\in\R\,:\, f_{vv}(t,v)-k g'''(v)=0\}|=0, \quad \forall\, k,t\ge0,
\end{equation}
where $|\{\cdots\}|$ denotes the one-dimensional Lebesgue measure of the set  $\{\cdots\}$.

Observe that from \eqref{e1.1000} and \eqref{e1.4''} it follows that 
$|\uu|^2\in H^1((0,T)\X\R)$, for all $T>0$.

\begin{definition}\label{D:1.1}  For all $T>0$, we say that 
$$
(\uu,v)\in L^2((0,T)\X\R;\C^2)\X  (L^1\cap L^\infty)((0,T)\X\R)
$$ 
is a weak solution of the problem \eqref{e1.1}--\eqref{e1.4} in $(0,T)\X\R$ if for all $\varphi\in 
C_c^\infty((-\infty,T)\X\R)$
the following holds
\begin{equation}\label{eD1.1} 
\begin{aligned}  
&\int_0^T\int_{\R} \frak u \varphi_t -\fka{\frak u}\varphi_x-i(\l U+\a g(v))\uu\varphi\,dx\,dt
+\int_\R \frak u_0 \varphi(0)\,dx=0,\\
&\int_0^T\int_{\R} v\varphi_t+(f(t,v)-\a g'(v)|\uu|^2)\varphi_x+h(t,v)\varphi\,dx\,dt + \int_{\R}v_0\varphi(0)\,dx=0.
\end{aligned}
\end{equation}
Moreover, for any convex 
$\eta\in C^2(\R)$,  we must have
\begin{multline}\label{e1.5'}
\eta(v)_t + \left(\int_0^v f_v(t,\xi) \eta'(\xi)\,d\xi -\a |\uu|^2\int_0^v \eta'(\xi)g''(\xi)\,d\xi   \right)_x+
 h(t,v)\eta'(v) \\ 
 \le \a  \left(\eta'(v)g'(v)-\int_0^v\eta'(\xi) g''(\xi)\,d\xi\right)(|\uu|^2)_x,
\end{multline}
in the sense of distributions in $(0,T)\X\R$.  
\end{definition}

We next state our existence result for the initial value problem \eqref{e1.1}--\eqref{e1.4}. 

\begin{theorem}\label{T:1.1} For all  $T>0$,  there exists a weak solution of the initial value problem \eqref{e1.1}-\eqref{e1.4} in $(0,T)\X\R$. 
Furthermore, if $(\uu^n, v^n)$ is a sequence 
of such weak solutions of the system \eqref{e1.1}--\eqref{e1.4}  with initial data 
$(\uu_0^n,v_0^n)$  uniformly bounded in $H^1(\R)\X (L^1\cap L^\infty)(\R)$,
converging in the sense of distributions to $(\uu_0,v_0)\in H^1(\R)\X (L^1\cap L^\infty)(\R)$
 then, by passing to a subsequence if necessary, $(\uu^n,v^n)$ converges in the sense of distributions  to a weak solution  of  \eqref{e1.1}--\eqref{e1.4}. 
\end{theorem}

\begin{proof}  By \eqref{e1.1000} we see that \eqref{e1.3}  essentially decouples from \eqref{e1.1}.
Therefore we can first solve \eqref{e1.3} and then plug the solution of \eqref{e1.3} into \eqref{e1.1}.
Further, since $U$ is also determined by the $\uu_0$, by Theorem~\ref{T:100}, we see that once we have the solution of \eqref{e1.3}, the solution of \eqref{e1.1} is immediate.  
We can use the vanishing viscosity method to solve \eqref{e1.3}. More specifically,  we approximate the solution of \eqref{e1.3} by solving the problem
\begin{align}
& v_t+\po_x f(t,v)+ h^\ve (t,v)=(g'(v)|\uu|^2)_x+\ve v_{xx}, \label{e1.41}\\
&v(0,x)=v_0^\ve(x), \label{e1.42}
\end{align}
where $h^\ve(t,v)=h(t,(1-\ve)v)$, $v_0^\ve= v_0*\rho_\ve$, where $\rho_\ve$ is a standard mollifying kernel,  for $0<\ve<<1$ such that $\pm h(t,\pm (1-\ve)c_0) >0$, respectively, which is possible by the hypotheses on $h$. Also, because of  the assumption on
$\supp g'$,  we can apply a standard maximum principle argument to deduce that the solution $v^\ve$ of \eqref{e1.41}-\eqref{e1.42} satisfies
\begin{equation}\label{e1.44}  
|v^\ve(t,x)|\le c_0,\quad \text{for all $(t,x)\in(0,\infty)\X\R$}.
\end{equation}
Using this a priori estimate, the solution of \eqref{e1.41}-\eqref{e1.42} follows easily by a  standard fixed point argument as explained in several text books, e.g., in \cite{Ho}, chapter~3 (see also \cite{Da}). 

Given $\eta\in C^2(\R)$ convex, multiplying \eqref{e1.41} by $\eta'$ and making trivial rearrangements we obtain
\begin{multline}\label{e1.5''}
\eta(v^\ve)_t + \left(\int_0^{v^\ve(t,x)} f_v(t,\xi) \eta'(\xi)\,d\xi -\a |\uu|^2\int_0^{v^\ve(t,x)} \eta'(\xi)g''(\xi)\,d\xi   \right)_x+
 h^\ve(t,v^\ve)\eta'(v^\ve) \\ 
 =\ve (\eta(v^\ve))_{xx}-\ve\eta''(v^\ve)|v_x^\ve|^2+ \a  \left(\eta'(v^\ve)g'(v^\ve)-\int_0^{v^\ve(t,x)}\eta'(\xi) g''(\xi)\,d\xi\right)(|\uu|^2)_x,
\end{multline}
Taking a strictly convex $\eta$, for instance, $\eta(v)=\frac12 v^2$, using the uniform boundedness \eqref{e1.41}  of $v^\ve$  and the fact that $(|\uu|^2)_x\in L^2((0,T)\X\R)$, for any $T>0$,  integrating \eqref{e1.5''} on  $(0,t)\X \R $,  we obtain 
\begin{multline}\label{e1.42'}
\int_\R (v^\ve(t))^2\,dx +\int_{(0,t)\X\R}\ve |v_x^\ve|^2\,dx\,dt\\ 
\le  \int_{\R}(v_0^\ve)^2\,dx+\int_{(0,t)\X\R}(|\uu|^2)_x^2\, dx\,dt
+ C\int_{(0,t)\X\R}|v^\ve|^2\,dx\,dt,
\end{multline}
which, by using Gronwall's inequality, gives the uniform boundedness of $v^\ve$ in $L^\infty((0,T); L^2(\R))$, for all $T>0$, and also  
\begin{equation}\label{e1.5'''}
 \int_{(0,T)\X\R}\ve |v_x^\ve|^2\,dx\,dt\le C(T),
 \end{equation}
 with $C(T)>0$ independent of $\ve$.  Let us denote
$$
q_\eta(t,v;|\uu|^2):=\int_0^{v^\ve(t,x)} f_v(t,\xi) \eta'(\xi)\,d\xi -\a |\uu|^2\int_0^{v^\ve(t,x)} \eta'(\xi)g''(\xi)\,d\xi.
$$  
{}From \eqref{e1.5'''}, it follows in a by now standard way that, for all $\eta\in C^2(\R)$,  
$$
\eta(v^\ve)_t+\po_x q_\eta(t,v;|\uu|^2)\in \text{compact in $W_\loc^{-1,2}((0,T)\X\R)$}.
$$  
Applying Tartar's compensated compactness argument in \cite{Ta}, using the non-degeneracy condition \eqref{e1.4'}, we obtain the convergence in $L_\loc^1((0,\infty)\X\R)$ of a subsequence of $v^\ve$,
also denoted $v^\ve$, to a function $v\in (L^1\cap L^\infty)((0,T)\X\R)$, for all $T>0$. The latter then clearly satisfies the second integral equation in \eqref{eD1.1} and \eqref{e1.5'}.   

We then use the obtained limit function $v$ in \eqref{e1.1} and also the fact that 
$U$ is determined by the initial data $\uu_0$, by Theorem~\ref{T:100}. We find a weak solution $\uu$ of \eqref{e1.1}, using Duhamel's principle,  by solving the integral equation 
\begin{equation}\label{e1.43'}
\uu(t)=S(t-t_0)\uu(t_0)-\int_{t_0}^t S(t-s)i(\l U(s)+\a g(v(s)))\uu(s)\,ds,
\end{equation}
where $S(t)$ is the unitary group generated by $\fka \frac{\po}{\po x}$, which is a skew-adjoint operator with domain $H^1(\R)$,  that is, $S(t)=\operatorname{exp}\left(\fka \frac{\po}{\po x}\right)t$. Using the fact that $(\l U(s)+\a g(v(s)))\in L^\infty((0,\infty)\X\R)$, we easily obtain a solution of \eqref{e1.43'} in $C([0,T_0];L^2(\R))$, with $t_0=0$,  by a fixed point  argument,  for $T_0>0$ small enough, whose smallness depends only on $\|(\l U+\a g(v))\|_\infty$. We then extend the solution using the same argument, for $t_0=T_0, 2T_0, 3T_0,\cdots$. By the semigroup property we then obtain a solution $\uu\in C([0,\infty); L^2(\R))$ to
\begin{equation}\label{e1.44'}
\uu(t)=S(t)\uu_0-\int_0^t S(t-s)i(\l U(s)+\a g(v(s)))\uu(s)\,ds,
\end{equation}
for all $t>0$. It is then standard to check that this solution of \eqref{e1.44'}  satisfies the first integral equation in \eqref{eD1.1}.

The second part of the statement follows by noticing that, under the assumptions in the statement, it follows that $(|\uu^n|^2)$, passing to a subsequence if necessary,  converges in $L^2((0,T)\X\R)$,   $v^n$ is bounded in $(L^1\cap L^\infty)((0,T)\X\R)$,
and the fact that the inequalities obtained from  \eqref{e1.5'} applied to $(\uu^n,v^n)$
imply, in a by now standard way, using the compactness of the embedding $\M_\loc((0,T)\X\R)\Subset W_\loc^{-1,p}((0,T)\X\R)$, for some $1<p<2$, where $\M_\loc((0,T)\X\R)$ is
the space of measures of locally finite variation,   and interpolation between 
$W_\loc^{-1,p}((0,T)\X\R)$ and $W^{-1,\infty}((0,T)\X\R)$ (see, e.g., \cite{Mu}),    that   
 $$
\eta(v^n)_t+\po_x q_\eta(t,v^n;|\uu^n|^2)\in \text{compact in $W_\loc^{-1,2}((0,T)\X\R)$}.
$$  
Therefore, we can apply again Tartar's compensated compactness arguments in \cite{Ta},
and the final assertion follows. 

 \end{proof}

\section{Application to the Augmented Born-Infeld Equations.}\label{S:4}

In this section we consider the interaction between short waves governed by a massless nonlinear Dirac equation  with long waves governed by the augmented Born-Infeld (ABI) equations, an extension of the  Born-Infeld equations introduced by Brenier in \cite{Br}, the latter 
being a nonlinear version of the Maxwell equations of the electromagnetism. 

The Born-Infeld equations ({\em cf.} \cite{Br}, see also, e.g., \cite{NS}) 
are obtained from the energy density $h$ given by
$$
h=\sqrt{1+B^2+D^2+|B\X D|^2}
$$
where $|\cdot|$ denotes the Euclidean norm, $B$ and $D$ are fields in $\R^3$ related with the magnetic
and electric  fields, $H$ and $E$, respectively, by the expressions
$$
E=\frac{\po h}{\po D}=\frac{D+B\X P}{h}, \quad H=\frac{\po h}{\po B}=\frac{B-D\X P}{h},
$$
where 
$$
P=D\X B
$$
is the Poynting vector.  The BI equations are 
\begin{equation}\label{e2.1}
 \begin{aligned}
 &\po_t D+\nabla\X \left(\frac{-B+ D\X P}{h}\right)=\po_t B+\nabla\X\left(\frac{D+ B\X P}{h}\right)=0 ,\\
 &\qquad\qquad \nabla \cdot D=\nabla \cdot B=0, 
 \end{aligned}
 \end{equation}
 and the energy density satisfies the additional conservation law
 \begin{equation}\label{e2.2}
 \po_t h+\nabla\cdot P=0.
 \end{equation}
 As remarked in \cite{Br},  $h$ is a strictly convex function of $B$ and $D$ only in a neighborhood 
 of the origin, not in the large, and it is not  clear that the BI equations  are hyperbolic in the large.
 Nevertheless, clearly $h$ is a global convex function of $B$, $D$ and $P=B\X D$. Motivated by this observation,  the following new evolution equation is obtained for $P$ in \cite{Br},
 \begin{equation}\label{e2.3}
 \po_t P+\nabla\cdot \left(\frac{P\otimes P-B\otimes B- D\otimes D}{h}\right) =\nabla\left(\frac1h \right).
 \end{equation}
 The 10$\X$10 system formed by the equations in \eqref{e2.1}, \eqref{e2.2} and \eqref{e2.3} is the so called augmented Born-Infeld (ABI) system. The hyperbolicity of the ABI system is proven in \cite{Br},
 where it is shown that 
 \begin{equation}\label{e2.4}
 S(D,B,P,h)=\frac{1+B^2+D^2+P^2}{2h}, \quad h>0,
 \end{equation}
 is convex entropy for the ABI system. More specifically, smooth solutions of the ABI system also satisfy the additional conservation law
 \begin{equation}\label{e2.5}
 \po_t S+ \nabla\cdot \frac{SP}{h}=\nabla\cdot \left\{\frac{P=D\X B+(B\cdot P)B+ (D\cdot P)D}{h^2}\right\}.
 \end{equation}
 Here, we are concerned with the plane waves of the  ABI system, that is, solutions that, with respect to the space variable $x=(x_1,x_2, x_3)$,  do not depend on $x_2, x_3$. Therefore,  it follows for these solutions that
 $$
 \po_t B_1=\po_t D_1=0,\quad \po_1 B_1=\po_1 D_1=0,
 $$
 which immediately follows from \eqref{e2.1}, which implies that $B_1$ and $D_1$ are constant.
 Let us define the positive constant $Z$ such that $Z^2=1+B_1^2+D_1^2$. The $8\X 8$ ABI system is as  follows:
 \begin{align}
 \po_t h+\po_1 P_1&=0, \label{e2.6}\\
 \po_t P_1+ \po_1\left(\frac{P_1^2-Z^2}{h}\right)&=0,\label{e2.7}\\
 \po_t D_2+\po_1\left(\frac{B_3+D_2P_1-D_1P_2}{h}\right)&=0,\label{e2.8}\\
 \po_t D_3+\po_1\left(\frac{-B_2+D_3P_1-D_1P_2}{h}\right)&=0,\label{e2.9}\\
 \po_t B_2+\po_1\left(\frac{-D_3+B_2P_1-D_1P_3}{h}\right)&=0,\label{e2.10}\\
 \po_t B_3+\po_1\left(\frac{D_2+B_3P_1-B_1P_3}{h}\right)&=0,\label{e2.11}\\
 \po_t P_2+\po_1\left(\frac{P_1P_2-D_1D_2-B_1B_2}{h}\right)&=0,\label{e2.12}\\
 \po_t P_3+\po_1\left(\frac{P_1P_3-D_1D_3-B_1B_3}{h}\right)&=0.\label{e2.13}
 \end{align}
We first observe that  \eqref{e2.6} and \eqref{e2.7} form a $2\X 2$ system decoupled from the remaining 6 equations of the ABI system. In fact, \eqref{e2.6}-\eqref{e2.7} describes the evolution of an isentropic gas  often called  Chaplygin  gas.  It is a linearly degenerate system, which is also the case of the whole ABI system. Since we want to describe the interaction of the long waves governed by the ABI system with 
short waves governed by a nonlinear Dirac equation, and the latter must be formulated in the Lagrangian coordinates of the long waves, we pass system \eqref{e2.6}--\eqref{e2.13} to Lagrangian coordinates as follows. Let us denote $\tau=\dfrac{1}{h}$, $v=\dfrac{P_1}{h}$, $\tilde D_i=\dfrac{D_i}{h}$,
$\tilde B_i=\dfrac{B_i}{h}$, $\tilde P_i=\dfrac{P_i}{h}$, $i=2,3$. We then get 
\begin{align}
 \po_t \tau-\po_y v&=0, \label{e2.14}\\
 \po_t v- \po_y\left(Z^2\tau \right)&=0,\label{e2.15}\\
 \po_t \tilde D_2+\po_y\left(\tilde B_3- D_1\tilde P_2\right)&=0,\label{e2.16}\\
 \po_t \tilde D_3-\po_y\left(\tilde B_2+ D_1\tilde P_3\right)&=0,\label{e2.17}\\
 \po_t \tilde B_2-\po_y\left(\tilde D_3+ B_1\tilde P_2\right)&=0,\label{e2.18}\\
 \po_t\tilde B_3+\po_y\left(\tilde D_2-  B_1\tilde P_3\right)&=0,\label{e2.19}\\
 \po_t \tilde P_2-\po_y\left( D_1\tilde D_2+ B_1\tilde B_2\right)&=0,\label{e2.20}\\
 \po_t \tilde P_3-\po_y\left( D_1\tilde D_3+ B_1\tilde B_3\right)&=0.\label{e2.21}
 \end{align}
 Hence, recalling that $D_1$ and $B_1$ are constants, we see that the plane waves of the ABI system are described by a linear hyperbolic system with constant coefficients in Lagrangian coordinates.
 Moreover, introducing the Riemann invariant variables $\theta=v+Z\tau$ and $\z=v-Z\tau$ equations \eqref{e2.14} and \eqref{e2.15} may be replaced by
 \begin{align}
 \po_t \theta-Z\po_y\theta&=0,\label{e2.14'}\\
 \po_t\z+Z\po_y \z&=0.\label{e2.15'}
 \end{align}
The physical region is $h\ge 1$, which is equivalent to $\theta-\z\le 2Z$. 

We propose to model the interaction of the electromagnetic waves governed by the ABI equations with short waves governed by a nonlinear Dirac equation by the following system
\begin{align}
& {\frak u}_t=\fka\, {\frak u}_y-i(\l U+\a_1 g_1(\z)+\a_2 g_2(\theta))\uu , \label{e2.22}\\
& \theta_t-Z \theta_y=\a_1 (g_1'(\theta)|\uu|^2)_y,\label{e2.24}\\
&\z_t+Z\z_y=\a_2 (g_2'(\z)|\uu|^2)_y,\label{e2.25}
\end{align}
where $U$ is as in \eqref{e1.00},  $g_1,g_2\in C^3(\R)$ such that for certain $a\le b \le c+2Z\le d+2Z$, $\supp g_1'\subset [a,b]$,
 $\supp g_2'\subset[c,d]$.
  
 Observe that in Lagrangian coordinates, the variables $\tilde D_i, \tilde B_i, \tilde P_i$, $i=2,3$,  are not affected by the interactions
 with the short waves and keep being described by the equations \eqref{e2.16}--\eqref{e2.21}. However, the corresponding variables in Eulerian coordinates, $D_i,B_i,P_i$, $i=2,3$,  are also affected by those interactions.
 More specifically, let us define 
 \begin{align*}
 \g_1(h,P_1)&=\frac{\a_1}{2Z}g_1'(\frac{P_1}h+\frac{Z}h)-\frac{\a_2}{2Z}g_2'(\frac{P_1}h-\frac{Z}h),\\
 \g_2(h,P_1)&=\frac{\a_1}{2}g_1'(\frac{P_1}h+\frac{Z}h)+\frac{\a_2}{2}g_2'(\frac{P_1}h-\frac{Z}h).
 \end{align*}   
 Then, in Eulerian coordinates, the ABI equations through the interactions with the short waves become
 \begin{align*}
 \po_t h+\po_1 P_1&=\po_1(\g_1(h,P_1)h|\uu|^2), \\
 \po_t P_1+ \po_1\left(\frac{P_1^2-Z^2}{h}\right)&=\po_1(\g_1(h,P_1) P_1 |\uu|^2)+\po_1(\g_2(h,P_1)|\uu|^2),\\
 \po_t D_2+\po_1\left(\frac{B_3+D_2P_1-D_1P_2}{h}\right)&=\po_1(\g_1(h,P_1)D_2|\uu|^2),\\
 \po_t D_3+\po_1\left(\frac{-B_2+D_3P_1-D_1P_2}{h}\right)&=\po_1(\g_1(h,P_1)D_3|\uu|^2),\\
 \po_t B_2+\po_1\left(\frac{-D_3+B_2P_1-D_1P_3}{h}\right)&=\po_1(\g_1(h,P_1)B_2|\uu|^2),\\
 \po_t B_3+\po_1\left(\frac{D_2+B_3P_1-B_1P_3}{h}\right)&=\po_1(\g_1(h,P_1)B_3|\uu|^2),\\
 \po_t P_2+\po_1\left(\frac{P_1P_2-D_1D_2-B_1B_2}{h}\right)&=\po_1(\g_1(h,P_1)P_2|\uu|^2),\\
 \po_t P_3+\po_1\left(\frac{P_1P_3-D_1D_3-B_1B_3}{h}\right)&=\po_1(\g_1(h,P_1)P_3
 |\uu|^2).
 \end{align*}
Thus, once we get a solution to \eqref{e2.22}--\eqref{e2.25}, together with a solution to \eqref{e2.16}--\eqref{e2.21}, we get, in particular, also a solution to the above forced ABI system with interaction forces in Eulerian coordinates, by using the inverse Lagrangian transformation, which is nonsingular in the region $h>0$.  Therefore, henceforth we will no longer mention the ABI system in Eulerian coordinates but only concentrate on solving the initial value problem for \eqref{e2.22}--\eqref{e2.25}.  
 
 We then prescribe the initial conditions
 \begin{equation} \label{e2.26}
\uu(0)=\uu_0,\quad \theta(0)=\theta_0,\quad \z(0)=\z_0, 
\end{equation}
with $\uu_0\in H^1(\R)$, $\theta_0,\z_0\in (L^1\cap L^\infty)(\R)$, so that
\begin{equation}\label{e2.27}
a<\theta_0 <b,\quad c<\z_0<d.
\end{equation}
We also assume the following non-degeneracy condition on $g_1,g_2$,
\begin{equation}\label{e2.28}
|\{\theta\in[a,b]\,:\, g_1'''(\theta)=0\}|=0=|\{\z\in[c,d]\,:\, g_2'''(\z)=0\}|.
\end{equation}

\begin{definition}\label{D:2.1}  For all $T>0$, we say that 
$$
(\uu,\theta,\z)\in L^2((0,T)\X\R;\C^2)\X  \left(L^\infty((0,T)\X\R)\right)^2
$$ 
is a weak solution of the problem \eqref{e2.22}--\eqref{e2.26} in $(0,T)\X\R$ if for all $\varphi\in 
C_c^\infty((-\infty,T)\X\R)$
the following holds
\begin{equation}\label{eD2.1} 
\begin{aligned}  
&\int_0^T\int_{\R} \frak u \varphi_t -\fka{\frak u}\varphi_x-i(\l U+\a_1 g_1(\theta)+\a_2g_2(\z) )\uu\varphi\,dx\,dt
+\int_\R \frak u_0 \varphi(0)\,dx=0,\\
&\int_0^T\int_{\R} \theta\varphi_t+(Z\theta-\a_1 g_1'(\theta)|\uu|^2)\varphi_x \,dx\,dt +
 \int_{\R}\theta_0\varphi(0)\,dx=0,\\
 &\int_0^T\int_{\R} \z\varphi_t+(Z\z-\a_2 g_2'(\z)|\uu|^2)\varphi_x \,dx\,dt +
 \int_{\R}\theta_0\varphi(0)\,dx=0.
\end{aligned}
\end{equation} 
 Moreover, for any convex $\eta\in C^2(\R)$,  we have
\begin{equation}\label{e2.entr}
\begin{aligned}
&\left(\eta(\theta)_t - \left(Z\eta(\theta) +\a_1 |\uu|^2\int_0^\theta \eta'(\xi)g_1''(\xi)\,d\xi   \right)_x \right)\Big\lfloor |\uu|^2 \\
&\qquad\qquad\le \left(\a_1  \left(\eta'(\theta)g_1'(\theta)-\int_0^\theta\eta'(\xi) g_1''(\xi)\,d\xi\right)(|\uu|^2)_x\right)\Big\lfloor |\uu|^2,\\
&\left(\eta(\zeta)_t + \left(Z\eta(\zeta) -\a_2 |\uu|^2\int_0^\zeta \eta'(\xi)g_2''(\xi)\,d\xi   \right)_x \right)\Big\lfloor |\uu|^2\\ 
&\qquad\qquad \le \left(\a_2  \left(\eta'(\zeta)g_2'(\zeta)-\int_0^\zeta\eta'(\xi) g_2''(\xi)\,d\xi\right)(|\uu|^2)_x\right)\Big\lfloor |\uu|^2,
\end{aligned}
\end{equation}
in the sense of the distributions where, for $\ell\in W_\loc^{-1,2}((0,T)\X\R)$,   $\ell\big\lfloor |\uu|^2\in  {\mathcal D}'((0,T)\X\R)$ is defined by $\la \ell\big\lfloor |\uu|^2,\varphi\ra=\la \ell, |\uu|^2\varphi\ra$,  for $\varphi\in C_c^\infty((0,T)\X\R)$.  
\end{definition} 

\begin{remark}\label{R:2.1} We remark that the restricted form of the entropy inequalities in \eqref{e2.entr} is due to the fact that the system \eqref{e2.22}--\eqref{e2.25} becomes linear where $|\uu|^2=0$. On the other hand, if $|\uu|^2>0$, the non-degeneracy condition \eqref{e2.28} ensures the nonlinear stability of the system, as we will see below. We also observe that, by the Remark~\ref{R:1.1}, if $\uu_0$ satisfies \eqref{e1.IC} then $|\uu|^2>0$ everywhere and the apparent restriction in  \eqref{e2.entr} is imaterial. 
\end{remark}

We then have the following theorem concerning the existence of a weak solution to the problem \eqref{e2.22}--\eqref{e2.26}.
 
\begin{theorem}\label{T:3.1} 
Given $T>0$,  for all $\a_1,\a_2>0$,  there exists a weak solution of the initial value problem \eqref{e2.22}-\eqref{e2.26}, $(\uu,\theta,\z)\in L^2([0,T]\X\bbT)\X  \left(L^\infty([0,T]\X\R)\right)^2$. 
Furthermore, if $(\uu^n,\theta^n, \z^n)$ is a sequence 
of such weak solutions of \eqref{e2.22}--\eqref{e2.26} with initial data 
$(\uu_0^n,\theta_0^n,\z_0^n)$  uniformly bounded in $H^1(\R)\X \left((L^1\cap L^\infty)(\R)\right)^2$, converging in the sense of distributions   to $(\uu_0,\theta_0, \z_0)\in H^1(\R)\X \left((L^1\cap L^\infty)(\R)\right)^2$
 then, by passing to a subsequence if necessary, $(\uu^n,\theta^n,\z^n)$ converges in the sense of distributions to a weak solution  of  \eqref{e2.22}--\eqref{e2.26}. 
\end{theorem}
 
 \begin{proof} 1. The proof is very similar to the proof of Theorem~\ref{T:1.1}, we only point out some points where the proof differs slightly from that one. Again,  we apply Theorem~\ref{T:100} from which it follows that $|\uu|^2$ is determined by $\uu_0$.
 So, first we approximate the solution of \eqref{e2.24}-\eqref{e2.25} by solving the problem
\begin{align}
& \theta_t-Z \theta_y=\a_1 (g_1'(\theta)|\uu|^2)_y+\ve \theta_{yy},\label{e2.71}\\
&\z_t+Z\z_y=\a_2 (g_2'(\z)|\uu|^2)_y + \ve\z_{yy},\label{e2.72}\\
&\theta(0,x)=\theta_0^\ve(x),\label{e2.73}\\
&\z(0,x)=\z_0^\ve(x),\label{e2.74}
 \end{align}
 where $\theta_0^\ve=\rho_\ve*\theta_0$, $\z_0^\ve=\rho_\ve*\z_0$, with $\rho_\ve$ as before. Denoting $\theta^\ve,\z^\ve$ the solution of \eqref{e2.71}-\eqref{e2.74}, by the assumption that $\supp g_1'\subset [a,b]$, $\supp g_2'\subset[c,d]$, and \eqref{e2.27},  using standard maximum principle arguments, we deduce the a priori estimate 
 \begin{equation}\label{e2.75}
 a\le \theta(t,x)\le b,\quad c\le \z(t,x)\le d.
 \end{equation}
 Again, for $\eta\in C^2(\R)$ convex we get
 \begin{multline}\label{e2.76}
\eta(\theta^\ve)_t - \left(\int_0^{\theta^\ve(t,y)} Z \xi \eta'(\xi)\,d\xi -\a_1 |\uu|^2\int_0^{\theta^\ve(t,y)} \eta'(\xi)g''(\xi)\,d\xi   \right)_y\\ 
 =\ve (\eta(\theta^\ve))_{yy}-\ve\eta''(\theta^\ve)|\theta_y^\ve|^2+ \a_1  \left(\eta'(\theta^\ve)g_1'(\theta^\ve)-\int_0^{\theta^\ve(t,y)}\eta'(\xi) g_1''(\xi)\,d\xi\right)(|\uu|^2)_y,
 \end{multline}
 \begin{multline}\label{e2.77}
 \eta(\z^\ve)_t +\left(\int_0^{\z^\ve(t,y)} Z \xi \eta'(\xi)\,d\xi -\a_2 |\uu|^2\int_0^{\z^\ve(t,y)} \eta'(\xi)g_2''(\xi)\,d\xi   \right)_y\\ 
 =\ve (\eta(\z^\ve))_{yy}-\ve\eta''(\z^\ve)|\z_y^\ve|^2+ \a_2  \left(\eta'(\z^\ve)g_2'(\z^\ve)-\int_0^{\z^\ve(t,y)}\eta'(\xi) g_2''(\xi)\,d\xi\right)(|\uu|^2)_y,
\end{multline}
 Again, for all $T>0$, and get  
\begin{equation}\label{e2.78}
 \int_{(0,T)\X\R}\ve (|\theta_y^\ve|^2+|\z^\ve_y|^2)\,dy\,dt\le C(T),
 \end{equation}
for some $C(T)>0$ independent of $\ve$. 
Denoting
\begin{align*}
&q_\eta^1(\theta;|\uu|^2):=\int_0^{\theta^\ve(t,y)} -Z \xi \eta'(\xi)\,d\xi -\a_1 |\uu|^2\int_0^{\theta^\ve(t,y)} \eta'(\xi)g_1''(\xi)\,d\xi,\\
&q_\eta^2(\z;|\uu|^2):=\int_0^{\z^\ve(t,y)} Z \xi \eta'(\xi)\,d\xi -\a_2 |\uu|^2\int_0^{\z^\ve(t,y)} \eta'(\xi)g_2''(\xi)\,d\xi,
\end{align*}  
Again, from \eqref{e2.78}, it follows,  for all $\eta\in C^2(\R)$,  
\begin{align*}
&\eta(\theta^\ve)_t+\po_x q_\eta^1(\theta^\ve;|\uu|^2)\in \text{compact in $W_\loc^{-1,2}((0,T)\X\R)$},\\
&\eta(\z^\ve)_t+\po_x q_\eta^2(\z^\ve;|\uu|^2)\in \text{compact in $W_\loc^{-1,2}((0,T)\X\R)$}.
\end{align*}
Again, applying Tartar's compensated compactness argument in \cite{Ta}, using the non-degeneracy condition \eqref{e2.28}, we obtain the convergence in the sense of distributions in $(0,\infty)\X\R)$ of a subsequence of $(\theta^\ve,\z^\ve)$,
also denoted $(\theta^\ve,\z^\ve)$, to a pair of functions 
$$
(\theta,\z)\in \left((L^1\cap L^\infty)((0,T)\X\R)\right)^2, \quad\text{for all $T>0$}.
$$
The convergence is in $L^1_\loc$ on the set $\{(t,y)\,:\, |\uu|^2>0\}$, where the Young measure generated by the referred subsequence reduces to a Dirac measure. With $(\theta,\z)$ at hand,
we solve the initial value problem for $\uu$ following the same procedures as in the last section.

  The second part of the statement also follows as in the last section by noticing that a subsequence of $(|\uu^n|^2)$  converges in $L^2((0,T)\X\R)$,  and  $(\theta^n,\z^n)$ is bounded in $(L^1\cap L^\infty)((0,T)\X\R)$,
and the fact that the inequalities obtained from  \eqref{e2.entr} applied to 
$(\uu^n,\theta^n,\z^n)$
imply, as explained in the last section,   that   
 \begin{align*}
&\eta(\theta^n)_t+\po_x q_\eta^1(\theta^n;|\uu^n|^2)\in \text{compact in $W_\loc^{-1,2}((0,T)\X\R)$},\\
&\eta(\z^n)_t+\po_x q_\eta^1(\z^n;|\uu^n|^2)\in \text{compact in $W_\loc^{-1,2}((0,T)\X\R)$}.
\end{align*} 
Therefore, as in the last section, we can apply again Tartar's compensated compactness arguments in \cite{Ta}, to conclude the proof of  the final assertion.

  \end{proof}

\section*{Acknowledgements} 

The first author would like to thank   Orfeu Bertolami, Jos\'e Pedro Mimoso and Vladimir Konotop  for encouraging comments about the model introduced in this paper.

\end{document}